\documentclass[11pt]{amsart}

\usepackage{amsmath}
\usepackage{amsfonts}
\usepackage{amssymb}
\usepackage{fullpage}
\usepackage{amscd}
\usepackage{amsthm}
\usepackage[numbers]{natbib} 
\usepackage{framed}
\usepackage{latexsym}

\usepackage{hyperref}
\hypersetup{colorlinks,linkcolor={red},citecolor={blue},urlcolor={blue}}

\evensidemargin\oddsidemargin

\newtheorem{theorem}{Theorem}[section]
\newtheorem{lemma}[theorem]{Lemma}

\newtheorem{proposition}[theorem]{Proposition}

\theoremstyle{definition}

\newtheorem{remark}[theorem]{Remark}

\newcommand{\CC}{\mathbb C}

\newcommand{\PP}{\mathbb P}

\newcommand{\RR}{\mathbb R}
\newcommand{\ZZ}{\mathbb Z}

\newcommand{\cD}{\mathcal D}

\newcommand{\SL}{\mathop{\mathrm {SL}}\nolimits}

\newcommand{\Orth}{\mathop{\null\mathrm {O}}\nolimits}

\newcommand{\rank}{\mathop{\mathrm {rank}}\nolimits}

\newcommand{\latt}[1]{{\langle{#1}\rangle}}

\newcommand{\m}{\operatorname{mod}}

\newcommand{\II}{\operatorname{II}}

\newcommand{\abs}[1]{\lvert#1\rvert}

\begin{document}

\title[Reflective modular forms on lattices of prime level]{Reflective modular forms on lattices of prime level}

\author{Haowu Wang}

\address{Center for Geometry and Physics, Institute for Basic Science (IBS), Pohang 37673, Korea}

\email{haowu.wangmath@gmail.com}


\date{\today}


\begin{abstract}
One of the main open problems in the theory of automorphic products is to classify reflective modular forms.  In \cite{Sch06} Scheithauer classified strongly reflective modular forms of singular weight on lattices of prime level.  In this paper  we classify symmetric reflective modular forms on lattices of prime level. This yields a full classification of lattices of prime level which have reflective modular forms. We also present some applications.
\end{abstract}

\maketitle

\section{Introduction}
Let $M$ be an even lattice of signature $(n,2)$ with $n\geq 3$ and $M^\vee$ be its dual lattice.  The type IV Hermitian symmetric domain $\cD(M)$ is defined as one of the two conjugate connected components of the space
\begin{equation*}
\{[\mathcal{Z}] \in  \PP(M\otimes \CC):  (\mathcal{Z}, \mathcal{Z})=0, (\mathcal{Z},\bar{\mathcal{Z}}) < 0\}.
\end{equation*}
Let $\Orth^+ (M) < \Orth(M)$ be the subgroup preserving $\cD(M)$. A  \textit{modular form} of weight $k\in \ZZ$ and character $\chi$ with respect to a finite index subgroup $\Gamma<\Orth^+ (M)$ is a holomorphic function on the affine cone of $\cD(M)$ satisfying
\begin{align*}
F(t\mathcal{Z})&=t^{-k}F(\mathcal{Z}), \quad \forall t \in \CC^*,\\
F(g\mathcal{Z})&=\chi(g)F(\mathcal{Z}), \quad \forall g\in \Gamma.
\end{align*}

The theory of automorphic Borcherds products (see \cite{Bor95, Bor98}) provides a remarkable method to construct modular forms whose divisor is a linear combination of rational quadratic divisors $\gamma^\perp:=\{ [v]\in \cD(M): (v,\gamma)=0\}$.  A modular form (for $\Gamma$) is called \textit{reflective} if its zeros are contained in the union of rational quadratic divisors $\gamma^\perp$ associated to roots of $M$, namely the reflection
$$
\sigma_\gamma: l\mapsto l-\frac{2(l,\gamma)}{(\gamma,\gamma)}\gamma
$$
belongs to $\Orth^+(M)$. A lattice is called \textit{reflective} if it has a reflective modular form.
Let $\widetilde{\Orth}^+(M)<\Orth^+ (M)$ denote the subgroup acting trivially on the discriminant form $M^\vee/M$. Bruinier's converse theorem says that every reflective modular form for $\widetilde{\Orth}^+(M)$ is a Borcherds product if $M$ splits two hyperbolic planes (see \cite{Bru02, Bru14}). Reflective modular forms have many applications in generalized Kac--Moody algebras, algebraic geometry, and reflection groups (see \cite{Bor00, GN98, Sch06, GHS07, GH14} and a survey \cite{Gri18}). It is an open problem to classify reflective modular forms and their underlying lattices since 1998. In the past twenty years, some classification results have been obtained in \cite{GN98, Sch06, Sch17, Ma17, Ma18, Dit18, Wan18, Wan19}.

In 2006 Scheithauer \cite{Sch06} classified all strongly reflective modular forms of singular weight on lattices of prime level and derived a classification of generalized Kac--Moody algebras. In the present paper,  we classify symmetric reflective modular forms on lattices of prime level without any restriction on the weight and the multiplicity of divisors.  Here,  a modular form is called symmetric if it is modular for $\Orth^+(M)$.  It is known that $M$ is reflective if and only if $M$ has a symmetric reflective modular form.  Thus the above classification implies a full classification of reflective lattices of prime level. 

Let $M$ be a lattice of prime level $p$.  Since $M$ and $M^\vee(p)$ have the same full orthogonal group,  $M$ is reflective if and only if $M^\vee(p)$ is reflective,  and a symmetric reflective modular form on $M$ corresponds to a symmetric reflective modular form on $M^\vee(p)$.  Therefore,  we only need to consider lattices of signature $(n,2)$, level $p$ and determinant $p^r$ with $1\leq r\leq 1+n/2$ (see \S \ref{Sec:basic lemmas}). The following is our main result.

\begin{theorem}\label{th:main}
All reflective lattices of genus $\II_{n,2}(p^{\epsilon_p n_p})$ with $1\leq n_p\leq 1+n/2$ are as follows
\begin{align*}
&\mathbf{p=2}& &\II_{6,2}(2_{\II}^{-2})& &\II_{6,2}(2_{\II}^{-4})& &\II_{10,2}(2_{\II}^{+2})& &\II_{10,2}(2_{\II}^{+4})& &\II_{10,2}(2_{\II}^{+6})& \\  
&& &\II_{14,2}(2_{\II}^{-2})& &\II_{14,2}(2_{\II}^{-4})& &\II_{14,2}(2_{\II}^{-6})& &\II_{14,2}(2_{\II}^{-8})& &\II_{18,2}(2_{\II}^{+2})&  \\
&& &\II_{18,2}(2_{\II}^{+4})& &\II_{18,2}(2_{\II}^{+6})& &\II_{18,2}(2_{\II}^{+8})& &\II_{18,2}(2_{\II}^{+10})& &\II_{22,2}(2_{\II}^{-2})& \\
&\mathbf{p=3}& &\II_{4,2}(3^{-1})& &\II_{4,2}(3^{+3})& &\II_{6,2}(3^{+2})& &\II_{6,2}(3^{-4})& &\II_{8,2}(3^{+1})& \\
&& &\II_{8,2}(3^{-3})& &\II_{8,2}(3^{+5})& &\II_{10,2}(3^{-2})& &\II_{10,2}(3^{+4})& &\II_{10,2}(3^{-6})& \\
&& &\II_{12,2}(3^{-1})& &\II_{12,2}(3^{+3})& &\II_{12,2}(3^{-5})& &\II_{12,2}(3^{+7})& &\II_{14,2}(3^{+2})&  \\
&& &\II_{14,2}(3^{-4})& &\II_{14,2}(3^{+6})& &\II_{14,2}(3^{-8})& &\II_{20,2}(3^{-1})&\\
&\mathbf{p=5}& &\II_{6,2}(5^{+1})& &\II_{6,2}(5^{-2})& &\II_{6,2}(5^{+3})& &\II_{6,2}(5^{-4})& &\II_{10,2}(5^{+2})& \\
&& &\II_{10,2}(5^{+4})& &\II_{10,2}(5^{+6})& &\II_{10,2}(5^{-1})&\\
&\mathbf{p=7}& &\II_{4,2}(7^{+1})& &\II_{4,2}(7^{-3})& &\II_{6,2}(7^{+2})& &\II_{6,2}(7^{-4})& &\II_{8,2}(7^{-1})&\\
&& &\II_{8,2}(7^{+3})& &\II_{8,2}(7^{-5})& \\
&\mathbf{p=11}& &\II_{4,2}(11^{-1})& &\II_{4,2}(11^{+3})& &\II_{6,2}(11^{+2})& &\II_{6,2}(11^{-4})&\\
&\mathbf{p=23}& &\II_{4,2}(23^{+1})& &\II_{4,2}(23^{-3})&
\end{align*}
The associated symmetric reflective modular forms are constructed in \S\ref{Sec:construction}.
\end{theorem}

Twelve of the above lattices appeared in \cite[Theorem 6.7]{Sch17} and Scheithauer constructed reflective modular forms of singular weight on them. When $M$ is of level one (i.e. unimodular), it is well-known that $M$ is reflective if the signature takes $(10,2)$, $(18,2)$ or $(26,2)$. By \cite{Ma17}, a lattice of signature $(n,2)$ is never reflective if $n > 26$. Thus there are exactly $3$ reflective unimodular lattices and the corresponding reflective modular forms are unique.  

We explain the main idea of the proof. Combining the arguments in \cite{Sch06, Dit18} and the Jacobi forms approach introduced in \cite{Wan18, Wan19} together, we are able to prove that every lattice not in the above list is not reflective. We use the quasi pull-back trick in \cite{BKP98}, the lifting from scalar-valued $\SL_2(\ZZ)$-modular forms to modular forms for the Weil representation in \cite{Bor00} and the obstruction principle in \cite{Bor99} to construct reflective modular forms.

This paper is organized as follows. In \S \ref{Sec:basic lemmas} we fix some notations and recall some known useful results about reflective modular forms. In \S \ref{Sec:Jacobi forms approach} we introduce the Jacobi forms approach. \S \ref{Sec:construction} is devoted to the construction of reflective modular forms. In \S \ref{Sec:proof} we prove our main theorem. In \S \ref{Sec:applications} we give some applications to the Kodaira dimension of orthogonal modular varieties. Two new examples of orthogonal modular varieties of Kodaira dimension $0$ and geometric genus $1$ are presented.
We also explain how to determine the class number of the genus of a lattice using our approach.

\section{Notations and basic lemmas}\label{Sec:basic lemmas}
Let $M$ be an even lattice of signature $(n,m)$ with bilinear form $(-,-)$. We denote the dual lattice of $M$ by $M^\vee$. The \textit{level} of $M$ is the smallest positive integer $N$ such that $N(x,x)\in 2\ZZ$ for all $x\in M^\vee$. The discriminant form of $M$ denoted $M^\vee / M$ can decompose into a sum of indecomposable Jordan components, and we denote this decomposition by $D_M$. The \textit{genus} of $M$ is the set of lattices which have the same signature and the same discriminant form (up to isomorphism) as $M$.  We use the standard notation $\II_{n,m}(D_M)$ to stand for the genus of $M$. We refer to \cite{CS99, Nik80} for the general theory and \cite[\S 3]{Sch06} for a brief introduction. 

 We now assume that $m=2$ and $n\geq 3$. 
 As in \cite[Lemma 2.2]{Ma17}, we show that if $M$ carries a reflective modular form for a finite-index subgroup $\Gamma < \Orth^+ (M)$ then $M$ also has a reflective modular form for any other finite-index subgroup $\Gamma' < \Orth^+ (M)$. Thus $M$ is reflective if and only if it has a reflective modular form for $\Orth^+(M)$. In this paper we only consider reflective modular forms with respect to $\Orth^+(M)$. Such modular forms are called \textit{symmetric} in Scheithauer's papers \cite{Sch06, Sch17}.  When the lattice $M$ has prime level,  Scheithauer established a bound on the signature of reflective Borcherds products not invariant under $\Orth^+(M)$ (see \cite[Theorem 6.5]{Sch17}).

We further assume that the signature $(n,2)$ lattice $M$ has prime level $p$. 
Then $M^\vee(p)$ is an even lattice of level $p$ or level one in which case it is unimodular. Since 
\begin{equation}\label{eq:O}
    \Orth^+(M)=\Orth^+(M^\vee)=\Orth^+(M^\vee(p)),
\end{equation}
a symmetric reflective modular form on $M$ can be viewed as a symmetric reflective modular form on $M^\vee(p)$. Thus $M$ is reflective if and only if $M^\vee(p)$ is reflective. 

In view of the above fact, throughout this paper we only consider lattices of genus $\II_{n,2}(p^{\epsilon_p n_p})$, where $n\geq 3$, $p$ is a prime number, $\epsilon_p=-$ or $+$, $1\leq n_p\leq n/2+1$.  By the Jordan decomposition of discriminant forms and the oddity formula (see \cite[\S 3]{Sch06}), $\epsilon_p$ is completely determined by $n$, $p$ and $n_p$. Therefore,  if two lattices of signature $(n,2)$ and prime level $p$ have the same determinant then they are isomorphic. Moreover, it is easy to derive the following facts
\begin{enumerate}
\item When $p=2$, the number $n-2$ is divisible by $4$ and $n_p$ is even;
\item When $p\equiv 1 \mod 4$, we have $n-2\in 4\ZZ$;
\item When $p\equiv 3 \mod 4$, we have $n\in 2\ZZ$. Moreover, if $n\equiv 0\mod 4$ then $n_p$ is odd; if $n\equiv 2\mod 4$ then $n_p$ is even.
\end{enumerate}

Let $M$ be such a lattice.  By \cite{Nik80}, $M$ can be represented as $U\oplus U(p)\oplus L$ or $2U\oplus L$, where $U$ is a hyperbolic plane, namely the unique even unimodular lattice of signature $(1,1)$, and $L$ is a positive definite lattice. A primitive vector $v\in M$
is reflective if and only if $(v,v)=2$ or $(v,v)=2p$ and $v/p \in M^\vee$ (see \cite[Proposition 2.5]{Sch06}).

By \cite[Proposition 5.1]{Sch15} and the Eichler criterion (see e.g. \cite[Proposition 4.1]{Gri18}), all vectors of norm $2$ in $M$ are in the same $\Orth^+(M)$-orbit, and all reflective vectors of norm $2p$ in $M$ are also in the same $\Orth^+(M)$-orbit. Therefore, for a symmetric reflective modular form, all $2$-reflective divisors have the same multiplicity denoted by $c_1$, and all $2p$-reflective divisors have the same multiplicity denoted by $c_p$.  A symmetric reflective modular form is called \textit{strongly} if it has only simple zeros (i.e. $c_1, c_p\leq 1$). A symmetric reflective modular form is called \textit{$2$-reflective} (resp. \textit{$2p$-reflective}) if $c_p=0$ (resp. $c_1=0$). A lattice $M$ is called \textit{$2$-reflective} (resp. \textit{$2p$-reflective}) if it has a $2$-reflective (resp. $2p$-reflective) modular form.

\begin{lemma}[Lemma 2.3 in \cite{Ma17}]\label{Lem:reductionMa}
If $M$ is $2$-reflective, then any even overlattice $M'$ of $M$ is also $2$-reflective.  
\end{lemma}

\begin{lemma}\label{Lem:reduction2p}
Let $M$ and $N$ be two lattices of signature $(n,2)$ and prime level $p$.
\begin{enumerate}
\item The lattice $M$ is $2p$-reflective if and only if $M^\vee(p)$ is $2$-reflective. 
\item Assume that $M$ is an overlattice of $N$. If $M$ is $2p$-reflective then $N$ is also $2p$-reflective.
\end{enumerate}
\end{lemma}
\begin{proof}
Let $F$ be a symmetric reflective modular form on $M$ with $c_1=a$ and $c_p=b$. Under the identification \eqref{eq:O}, $F$ defines a symmetric reflective modular form on $M^\vee(p)$ with $c_1=b$ and $c_p=a$. This implies the first assertion of the lemma. The second assertion follows from the first assertion by applying Lemma \ref{Lem:reductionMa} to $M^\vee(p) < N^\vee(p)$. 
\end{proof}

The following result proved by Scheithauer gives a bound on the signature of $2p$-reflective Borcherds products.

\begin{lemma}[Proposition 6.1 in \cite{Sch17}]\label{Lem:Sch2p}
Let $M$ be a lattice of signature $(n,2)$ and prime level $p$.
If $M$ has a $2p$-reflective modular form which can be constructed as a Borcherds product of a vector-valued modular form associated to $M^\vee/M$, then $n\leq 2+24/(p+1)$.
\end{lemma}

We introduce a special case of Dittmann's result. We repeat its proof, because not only the result is important to us, but also we can derive one construction of reflective modular forms (see Lemma \ref{rem:Dit} below) from the proof. 

\begin{lemma}[Lemma 4.5 in \cite{Dit18}]\label{Lem:Dit}
Let $M$ be a reflective lattice of signature $(n,2)$ and prime level $p$. If $M$ can be expressed as $U\oplus U(p)\oplus L$, then $n\leq 2+48/(p+1)$. 
\end{lemma}

\begin{proof}
Suppose that $F$ is a reflective modular form for $\Orth^+(M)$. Since $M=U\oplus U(p)\oplus L$, we derive from \cite{Bru14} that $F$ is a Borcherds product. By \cite[Corollary 5.5]{Sch15}, the corresponding vector-valued modular form can be constructed as a lifting of some modular form for $\Gamma_0(p)$. More precisely, there exists a nearly holomorphic modular form $f$ of weight $1-n/2$ for $\Gamma_0(p)$ with a character. We write $f=aq^{-1}+b+O(q)$. We see from the expression of $M$ that there are nontrivial vectors of norm $0$ in the discriminant group of $M$. By \cite[Theorem 6.2]{Sch06}, we conclude that $f\vert_S=c q^{-1/p}+d + O(q^{1/p})$, otherwise there will be some principal Fourier coefficients which give non-reflective divisors. Here $a$, $b$, $c$, $d$ are constants. The Riemann--Roch theorem applied to $f$ gives
$$
-2\leq p\nu_0(f)+\nu_\infty(f) \leq \frac{p+1}{12}\left(1-\frac{n}{2}\right),
$$
which proves the lemma.
\end{proof}

\begin{lemma}\label{rem:Dit}
Let $M=U\oplus U(p)\oplus L$ and $M_1=2U\oplus L$. Assume that $F$ is a reflective modular form for $\Orth^+(M)$ with multiplicities $c_1\neq 0$ and $c_p$.  Then $M_1$ has a symmetric reflective modular form whose two types of multiplicities are respectively $c_1$ and $p c_p$.
\end{lemma}
\begin{proof}
By the proof of Lemma \ref{Lem:Dit}, the modular form $F$ is the Borcherds product of a certain lifting of $f$ for $M^\vee / M$. By calculating the lifting of $f$ for $M_1^\vee/ M_1$ (see \cite[Theorem 6.2]{Sch06} for an explicit definition), we find that its Borcherds product 
gives a reflective modular form for $\Orth^+(M_1)$ with multiplicities $c_1$ and $pc_p$. 
\end{proof}

Using the idea of pull-backs, we prove the following result (see \cite[\S 5]{Wan19}).
\begin{lemma}\label{Lem:pullbacktest}
Let $M$ be an even lattice of signature of $(n,2)$ and $L$ be an even positive definite lattice. If $M\oplus L$ is reflective (resp. $2$-reflective), then $M$ is reflective (resp. $2$-reflective) too.
\end{lemma}

In the lemma above, the reflective modular form for $M$ is constructed as the quasi pull-back $M \hookrightarrow M\oplus L$. The pull-back trick was first used in \cite{BKP98}.  We refer to \cite[Theorem 6.1]{Gri18} for a detailed description. Let $F$ be a reflective modular form of weight $k$ for $M\oplus L$. Then the weight of the quasi pull-back is given by $k$ plus one half of the number of divisors of $F$ contained in $L$ (counting multiplicity).

\section{The Jacobi forms approach}\label{Sec:Jacobi forms approach}
In \cite{Wan18, Wan19}, we developed a new approach based on the theory of Jacobi forms to classify reflective modular forms. We describe this approach in the case of prime level. We refer to \cite{EZ85, Gri18} for the theory of Jacobi forms.

\begin{proposition}\label{Prop:Jacobi}
Let $M=2U\oplus L$ be a lattice of prime level $p$ and $F$ be a reflective modular form of weight $k$ for $\Orth^+(M)$. Assume that the $2$-reflective and $2p$-reflective divisors of $F$ have multiplicities $c_1$ and $c_p$, respectively. We define the root system associated to $L$ as $R(L)=R_1(L)\cup R_2(L)$, where
$$
R_1(L)=\{ v\in L: (v,v)=2 \}, \quad R_2(L)=\{ v\in L: (v,v)=2p, v/p \in L^\vee \}.
$$
\begin{enumerate}
\item If $R(L)$ is empty, then $k=12c_1$.

\item If $R(L)$ is non-empty, then $R(L)$ generates $L\otimes \RR$ and thus it is an usual root system of rank equal to $\rank(L)$.  Moreover, the following identities hold
\begin{equation}\label{eq:one} 
\begin{split}
C:=&\frac{1}{24}\left(c_1 \abs{R_1(L)}+ c_p \abs{R_2(L)}+ 2k \right)- c_1 \\
= &\frac{1}{2\rank(L)}\left(2c_1 \abs{R_1(L)}+ \frac{2}{p}c_p \abs{R_2(L)} \right),
\end{split}
\end{equation}
\begin{equation}\label{eq:two} 
c_1 \sum_{r\in R_1(L)} (r,\mathfrak{z})^2 + c_p\sum_{s\in R_2(L)} (s/p,\mathfrak{z})^2 = 2C (\mathfrak{z},\mathfrak{z}), \quad \mathfrak{z}\in L\otimes \CC.
\end{equation}

\item When $p\geq 5$, as lattices, we have $R_1(L)\cap R_2(L) =\emptyset$, i.e. $R(L)=R_1(L)\oplus R_2(L)$. The set $R_1(L)$ is a direct sum of some $ADE$-type root systems and all irreducible components have the same Coxeter number denoted by $h_1$. The set $R_2(L)$ can be represented as the $p$-rescaling $R_2(p)$, where $R_2$ is a direct sum of some $ADE$-type root systems and all irreducible components have the same Coxeter number denoted by $h_2$. Let $n_1$ be the rank of $R_1(L)$. If $R_1(L)\neq \emptyset$ and $R_2(L)\neq \emptyset$, then $c_p\neq 0$ and we have 
\begin{align*}
C&=c_1h_1 = \frac{c_p h_2}{p},\\
k&= c_1\left[ 12(h_1+1)+\frac{1}{2}\left( p -1 \right)n_1h_1 - \frac{1}{2}\rank(L) ph_1   \right].
\end{align*}
Moreover, we have 
\begin{equation}\label{eq:singular}
k\geq \frac{1}{2}\left[ n_1 c_1 + (\rank(L)-n_1)c_p \right].
\end{equation}
\end{enumerate}
\end{proposition}

\begin{proof}
We know from \cite{Bru14} that $F$ is a Borcherds product. In view of the isomorphism between the spaces of vector-valued modular forms and Jacobi forms, there exists a weakly holomorphic Jacobi form $\phi$ of weight $0$ and index $L$ whose Borcherds product gives $F$ (see \cite{Wan19}). The divisors of $F$ of the form $(0,0,x,1,0)^\perp$ determine the $q^0$-term of $\phi$ (see \cite[Theorem 4.2]{Gri18}). More precisely, we have
$$
\phi=c_1q^{-1} + c_1\sum_{r\in R_1(L)} e^{2\pi i(r,\mathfrak{z})} + c_p\sum_{s\in R_2(L)} e^{2\pi i(s/p,\mathfrak{z})} +2k +O(q).
$$
It was proved in \cite[Proposition 2.6]{Gri18} that the $q^0$-term of a Jacobi form of weight $0$ satisfies two relations, which are exactly the two identities of the assertion (2) in our case. When $R(L)=\emptyset$, we have $C=0$ and then $\frac{1}{24}\times 2k - c_1 =0$, which yields $k=12c_1$. When $R(L)\neq\emptyset$, we deduce from \eqref{eq:two} that $R(L)$ generates $L\otimes \RR$, otherwise there will be a nonzero vector orthogonal to $R(L)$ and then $C=0$, which leads to a contradiction. It is easy to check that $R(L)$ is an usual root system by definition. 

We now prove the last assertion. It is well-known that every root system is a direct sum of some irreducible root systems of type $A_n$, $B_n$, $C_n$, $D_n$, $E_6$, $E_7$, $E_8$, $G_2$, $F_4$ and their rescalings (see \cite{Bou60}). But there are only $2$- and $2p$-reflections in $R(L)$. When $p\geq 5$, $R_1(L)$ must be a direct sum of $ADE$-type root systems and $R_2(L)$ must be a direct sum of $p$-rescalings of $ADE$-type root systems. It follows that $R_1(L)\cap R_2(L) =\emptyset$. We conclude from \eqref{eq:two} that the irreducible components of $R_1(L)$ and $R_2(L)$ have the same Coxeter number respectively, and also $C=c_1h_1=c_ph_2/p$. The identity related to weight $k$ follows from \eqref{eq:one}. To prove  inequality \eqref{eq:singular}, we notice that the $q^0$-term of $\phi$ defines a holomorphic Jacobi form of weight $k$ for $L$ as a theta block, which is also a holomorphic Jacobi form for $R_1(L)\oplus R_2(L)$ (see \cite[Theorem 4.2]{Gri18} or \cite[Theorem 4.6]{Wan19}). More precisely, this Jacobi form has the form (see \cite[Corollary 2.7]{Gri18} for this expression and notations):
\begin{equation}\label{eq:theta}
    \eta^{2k}(\tau) \prod_{r\in R_1(L), r>0} \left( \frac{\vartheta(\tau, (r,\mathfrak{z}))}{\eta(\tau)} \right)^{c_1} \prod_{s\in R_2(L), s>0} \left( \frac{\vartheta(\tau, (s/p,\mathfrak{z}))}{\eta(\tau)} \right)^{c_p}. 
\end{equation}
Since $R_1(L)$ is orthogonal to $R_2(L)$,  we see from the above expression that this Jacobi form is the product of two holomorphic Jacobi forms for $R_1(L)$ and $R_2(L)$.  The singular weight argument of  Jacobi forms says that if a holomorphic Jacobi form of index $L$ is not constant then its weight is not less than $\rank(L)/2$ (see \cite[Page 823]{Gri18}).  From this fact, we conclude the desired inequality. We then finish the proof of the result.
\end{proof}

\section{The construction of reflective modular forms}\label{Sec:construction}
In this section we construct reflective modular forms for all lattices listed in Theorem \ref{th:main}.
\begin{theorem}\label{th:construction1}
The following lattices have strongly symmetric $2$-reflective modular forms. We give a model for every lattice and indicate the weight $k$ of the modular form.
\begin{align*}
&\II_{6,2}(2_{\II}^{-2}), 2U\oplus D_4, k=72& &\II_{6,2}(2_{\II}^{-4}), U\oplus U(2)\oplus D_4, k=40&\\
&\II_{10,2}(2_{\II}^{+2}), 2U\oplus D_8, k=124&  &\II_{10,2}(2_{\II}^{+4}), 2U\oplus 2D_4, k=60& \\
&\II_{10,2}(2_{\II}^{+6}), 2U\oplus D_8^\vee(2), k=28& &\II_{4,2}(3^{-1}), 2U\oplus A_2, k=45& \\
&\II_{4,2}(3^{+3}), U\oplus U(3)\oplus A_2, k=18& &\II_{6,2}(3^{+2}), 2U\oplus 2A_2, k=42&\\
&\II_{6,2}(3^{-4}), U\oplus U(3)\oplus 2A_2, k=15& &\II_{8,2}(3^{+1}), 2U\oplus E_6, k=120&\\
&\II_{8,2}(3^{-3}), 2U\oplus 3A_2, k=39& &\II_{8,2}(3^{+5}), 2U\oplus E_6^\vee(3), k=12&\\
&\II_{6,2}(5^{+1}), 2U\oplus A_4, k=62&  &\II_{6,2}(5^{+3}), 2U\oplus A_4^\vee(5), k=12&\\
&\II_{8,2}(7^{-1}), 2U\oplus A_6, k=75& 
\end{align*}
\end{theorem}

\begin{proof}
The 2-reflective modular form for $\II_{6,2}(5^{+3})$ was constructed in \cite{GW19}.  The 2-reflective modular forms for $\II_{4,2}(3^{+3})$, $\II_{6,2}(3^{-4})$ and $\II_{8,2}(3^{+5})$ can be constructed using the following embedding of lattices
$$
U\oplus U(3)\oplus A_2 \hookrightarrow U\oplus U(3)\oplus 2A_2 \hookrightarrow U\oplus U(3)\oplus 3A_2 \cong 2U\oplus E_6^\vee(3).
$$
By  Lemma \ref{Lem:pullbacktest}, it suffices to show that $2U\oplus E_6^\vee(3)$ is $2$-reflective. Notice that 
\begin{equation}\label{towerA2}
2U\oplus E_6^\vee(3)\cong U\oplus U(3)\oplus 3A_2 \hookleftarrow 2U(3)\oplus 3A_2 = (2U\oplus 3A_2)^\vee (3).
\end{equation}
There is a $6$-reflective modular form of singular weight $3$ for $2U\oplus 3A_2$ in \cite{Sch06}. By Lemma \ref{Lem:reduction2p}, the lattice $2U(3)\oplus 3A_2$ is 2-reflective. We obtain by Lemma \ref{Lem:reductionMa} that $2U\oplus E_6^\vee(3)$ is also $2$-reflective. Since $E_6^\vee(3)$ has no 2-roots, we deduce from Proposition \ref{Prop:Jacobi} that the 2-reflective modular form for $2U\oplus E_6^\vee(3)$ has weight $12$. Note that the weight of a quasi pull-back is easy to work out. For example, $A_2$ has six 2-roots. Thus the $2$-reflective modular form for $U\oplus U(3)\oplus 2A_2$ has weight $12+\frac{1}{2}\times 6=15$. The $2$-reflective modular forms for other lattices were constructed in \cite[\S 4]{GN18} as quasi pull-backs of the Borcherds form $\Phi_{12}$ for the even unimodular lattice of signature $(26,2)$. In general, a quasi pull-back is only modular for the discriminant kernel $\widetilde{\Orth}^+(M)$. However, when the lattice is of prime level, the set of $2$-reflective vectors has only one orbit under the action of $\widetilde{\Orth}^+(M)$. Thus the symmetrization of a quasi pull-back coincides with some power of the quasi pull-back. Therefore, these constructed forms are modular under $\Orth^+(M)$ and have only simple zeros. 
\end{proof}

\begin{theorem}
The following lattices have strongly symmetric $2p$-reflective modular forms with indicated weight $k$.
\begin{align*}
&\II_{6,2}(2_{\II}^{-2}), 2U\oplus D_4, k=24& &\II_{6,2}(2_{\II}^{-4}), U\oplus U(2)\oplus D_4, k=40&\\
&\II_{10,2}(2_{\II}^{+2}), 2U\oplus D_8, k=4&  &\II_{10,2}(2_{\II}^{+4}), 2U\oplus 2D_4, k=12& \\
&\II_{10,2}(2_{\II}^{+6}), 2U\oplus D_8^\vee(2), k=28& &\II_{4,2}(3^{-1}), 2U\oplus A_2, k=9& \\
&\II_{4,2}(3^{+3}), U\oplus U(3)\oplus A_2, k=18& &\II_{6,2}(3^{+2}), 2U\oplus 2A_2, k=6&\\
&\II_{6,2}(3^{-4}), U\oplus U(3)\oplus 2A_2, k=15& &\II_{8,2}(3^{-3}), 2U\oplus 3A_2, k=3&\\ &\II_{8,2}(3^{+5}), 2U\oplus E_6^\vee(3), k=12& &\II_{6,2}(5^{+3}), 2U\oplus A_4^\vee(5), k=2&
\end{align*}
\end{theorem}

\begin{proof}
These $2p$-reflective modular forms are also constructed using quasi pull-backs of some known reflective modular forms. Some of them can also be constructed as additive liftings of Jacobi forms (see \cite[\S 5]{Gri18}).  We only consider one tower. Since $2U\oplus E_6^\vee(3)= (2U\oplus E_6^\vee(3))^\vee(3)$, we have by Lemma \ref{Lem:reduction2p} that $2U\oplus E_6^\vee(3)$ has a $6$-reflective modular form of weight $12$. Then the quasi pull-backs related to the tower  \eqref{towerA2} give 6-reflective modular forms of weight $15$ and $18$ for $U\oplus U(3)\oplus 2A_2$ and $U\oplus U(3)\oplus A_2$, respectively. Again, a quasi pull-back is usually not invariant under $\Orth^+(M)$,  but its symmetrization will give a symmetric $2p$-reflective modular form whose weight can be determined by Proposition \ref{Prop:Jacobi}. 
\end{proof}

\begin{theorem}\label{th:construction3}
The following lattices have symmetric reflective modular forms with multiplicity $c_1=1$. These modular forms cannot be decomposed into a product of $2$-reflective and $2p$-reflective modular forms. We indicate their weight $k$, multiplicity $c_p$, and whether they are cusp forms.
\begin{align*}
&\II_{14,2}(2_{\II}^{-2}),&  &2U\oplus E_8\oplus D_4, &  & k=144, &  & c_2=8, &  & \text{cusp}\\
&\II_{14,2}(2_{\II}^{-4}),&  &2U\oplus D_8\oplus D_4, &  & k=80, &  & c_2=4 , &  & \text{cusp}\\
&\II_{14,2}(2_{\II}^{-6}),&  &2U\oplus 3D_4, &  & k=48,&  & c_2=2 , &  & \text{cusp}\\
&\II_{14,2}(2_{\II}^{-8}),& &2U\oplus D_8^\vee(2)\oplus D_4, &  & k=32,&  & c_2=1, &  & \text{cusp}\\
&\II_{18,2}(2_{\II}^{+2}),& &2U\oplus E_8\oplus D_8, &  & k=68,&  & c_2=16, &  & \text{cusp}
\\
&\II_{18,2}(2_{\II}^{+4}),& &2U\oplus E_8\oplus 2D_4, &  & k=36, &  & c_2=8, &  & \text{cusp}\\
&\II_{18,2}(2_{\II}^{+6}),& &2U\oplus E_8\oplus D_8^\vee(2), &  & k=20,&  & c_2=4, &  & \text{cusp}\\
&\II_{18,2}(2_{\II}^{+8}),& &2U\oplus E_8\oplus E_8(2),&  & k=12, &  & c_2=2, &  & \text{non-cusp}\\
&\II_{18,2}(2_{\II}^{+10}),& &2U\oplus D_8\oplus E_8(2),&  & k=8,&  & c_2=1, &  & \text{non-cusp}\\
&\II_{22,2}(2_{\II}^{-2}),&  &2U\oplus 2E_8\oplus D_4,&  & k=24,&  & c_2=8, &  & \text{cusp}\\
&\II_{10,2}(3^{-2}),& & 2U\oplus E_6\oplus A_2,& & k=90,& & c_3=9,&  & \text{cusp}\\
&\II_{10,2}(3^{+4}),& &  2U\oplus 4A_2, & & k=36, & & c_3=3,&  & \text{cusp}\\
&\II_{10,2}(3^{-6}),& & 2U\oplus E_6^\vee(3)\oplus A_2,& & k=18,& & c_3=1,&  & \text{cusp}\\
&\II_{12,2}(3^{-1}),& & 2U\oplus E_8\oplus A_2,& & k=168,& & c_3=27,&  & \text{cusp}\\
&\II_{12,2}(3^{+3}),& & 2U\oplus E_6\oplus 2A_2,& & k=60,& & c_3=9,&  & \text{cusp}\\
&\II_{12,2}(3^{-5}),& & 2U\oplus 5A_2,& & k=24,& & c_3=3,&  & \text{cusp}\\
&\II_{12,2}(3^{+7}),& & 2U\oplus E_6^\vee(3)\oplus 2A_2,& & k=12,& & c_3=1,&  & \text{cusp}\\
&\II_{14,2}(3^{+2}),& & 2U\oplus E_8\oplus 2A_2,& & k=84,& & c_3=27,&  & \text{cusp}
\end{align*}

\begin{align*}
&\II_{14,2}(3^{-4}),& & 2U\oplus E_6\oplus 3A_2,& & k=30,& & c_3=9,&  & \text{cusp}\\
&\II_{14,2}(3^{+6}),& & 2U\oplus 6A_2,& & k=12,& & c_3=3,&  & \text{non-cusp}\\
&\II_{14,2}(3^{-8}),& & 2U\oplus E_6^\vee(3)\oplus 3A_2,& & k=6,& & c_3=1,&  & \text{non-cusp}\\
&\II_{20,2}(3^{-1}),& & 2U\oplus 2E_8\oplus A_2,& & k=48,& & c_3=27,&  & \text{cusp}\\
&\II_{6,2}(5^{-2}),& & 2U\oplus T_4,& & k=30,& & c_5=5,&  & \text{cusp}\\
&\II_{6,2}(5^{-4}),& & U\oplus U(5)\oplus T_4,& & k=10,& & c_5=1,&  & \text{cusp}\\
&\II_{10,2}(5^{+2}),& & 2U\oplus 2A_4,& & k=52,& & c_5=25,&  & \text{cusp}\\
&\II_{10,2}(5^{+4}),& & 2U\oplus A_4^\vee(5)\oplus A_4,& & k=12,& & c_5=5,&  & \text{non-cusp}\\
&\II_{10,2}(5^{+6}),& & 2U\oplus 2A_4^\vee(5),& & k=4,& & c_5=1,&  & \text{non-cusp}\\
&\II_{10,2}(5^{-1}),& & 2U\oplus T_8,& & k=120,& & c_5=45,&  & \text{cusp}\\
&\II_{4,2}(7^{+1}),& & 2U\oplus L_7,& & k=28,& & c_7=7, & &\text{cusp}\\
&\II_{4,2}(7^{-3}),& & U\oplus U(7)\oplus L_7,& & k=7,& & c_7=1,& &\text{cusp}
\\
&\II_{6,2}(7^{+2}),& & 2U\oplus 2L_7,& & k=20,& & c_7=7,& &\text{cusp}\\
&\II_{6,2}(7^{-4}),& & U\oplus U(7)\oplus 2L_7,& & k=5,& & c_7=1,& &\text{cusp}\\ 
&\II_{8,2}(7^{+3}),& & 2U\oplus 3L_7,& & k=12,& & c_7=7,& &\text{non-cusp}\\
&\II_{8,2}(7^{-5}),& & 2U\oplus A_6^\vee(7),& & k=3,& & c_7=1,& & \text{non-cusp}\\
&\II_{4,2}(11^{-1}), & & 2U\oplus L_{11},& & k=24,& & c_{11}=11,& &\text{cusp}\\
&\II_{4,2}(11^{+3}), & & U\oplus U(11)\oplus L_{11},& & k=4,& & c_{11}=1,& &\text{cusp}\\
&\II_{6,2}(11^{+2}),& & 2U\oplus 2L_{11},& & k=12,& & c_{11}=11,& &\text{non-cusp}\\
&\II_{6,2}(11^{-4}), & & U\oplus U(11)\oplus 2L_{11},& & k=2,& & c_{11}=1,& &\text{non-cusp}\\
&\II_{4,2}(23^{+1}),& &  2U\oplus L_{23},& & k=12,& & c_{23}=23,& &\text{non-cusp}\\
&\II_{4,2}(23^{-3}),& &  U\oplus U(23)\oplus L_{23},& & k=1,& & c_{23}=1,& &\text{non-cusp}
\end{align*}
The lattices $L_7$, $L_{11}$ and $T_4$ are defined as follows
\begin{align*}
&L_7=\left( \begin{array}{cc}
2 & 1 \\ 
1 & 4
\end{array}  \right),& &L_{11}=\left( \begin{array}{cc}
2 & 1 \\ 
1 & 6
\end{array}  \right),&\\
&L_{23}=\left( \begin{array}{cc}
2 & 1 \\ 
1 & 12
\end{array}  \right),& &T_4=\left( \begin{array}{cccc}
2 & 1 & 1 & 1 \\ 
1 & 2 & 0 & 1 \\ 
1 & 0 & 4 & 2 \\ 
1 & 1 & 2 & 4
\end{array}  \right).&
\end{align*}
The lattice $T_8$ is a nontrivial even overlattice of $E_7\oplus A_1(5)$ and thus it has determinant $5$.
\end{theorem}

In general, a lattice of signature $(n,2)$ may have several models, for example
\begin{align*}
&U\oplus U(2)\oplus \text{Barnes-Wall lattice}  \cong U\oplus U(2)\oplus E_8\oplus E_8(2) \\
\cong &  U\oplus U(2)\oplus 4D_4 \cong U\oplus U(2)\oplus D_8\oplus D_8^\vee(2)
\cong  2U\oplus E_8(2)\oplus D_8\\
\cong &  2U(2)\oplus D_8\oplus 2D_4 \cong 2U\oplus D_8^\vee(2)\oplus 2D_4\cong 2U(2)\oplus E_8\oplus D_8^\vee(2).
\end{align*}
In the above theorem, we prefer to give models of type $2U\oplus L$.

\begin{proof}
We have the following embeddings of lattices 
\begin{align*}
&p=2& &U\oplus U(2)\oplus 4D_4 \hookleftarrow U\oplus U(2)\oplus 3D_4 \hookleftarrow U\oplus U(2)\oplus 2D_4 \hookleftarrow U\oplus U(2)\oplus D_4,\\
&p=3&  &U\oplus U(3)\oplus 6A_2 \hookleftarrow U\oplus U(3)\oplus 5A_2 \hookleftarrow U\oplus U(3)\oplus 4A_2 \hookleftarrow U\oplus U(3)\oplus 3A_2  \\
&& &\hookleftarrow U\oplus U(3)\oplus 2A_2 \hookleftarrow U\oplus U(3)\oplus A_2,\\
&p=5& &U\oplus U(5)\oplus 2T_4 \hookleftarrow U\oplus U(5)\oplus T_4,\\
&p=7& &U\oplus U(7)\oplus 3L_{7} \hookleftarrow U\oplus U(7)\oplus 2L_{7} \hookleftarrow U\oplus U(7)\oplus L_{7},\\
&p=11& &U\oplus U(11)\oplus 2L_{11} \hookleftarrow U\oplus U(11)\oplus L_{11}.
\end{align*}
For the first lattice in every tower, Scheithauer \cite{Sch06} constructed a strongly reflective modular form of singular weight. We then construct many reflective modular forms using quasi pull-backs. For some other lattices, we use Lemma \ref{rem:Dit} to build reflective modular forms. For example, we construct the reflective modular form for $\II_{18,2}(2_{\II}^{+n_2})$, where $n_2=2,4,6,8$. Firstly, by \cite{Sch06}, there is a strongly reflective modular form of singular weight $8$ for $\II_{18,2}(2_{\II}^{+10})$. This modular form is a Borcherds product of a vector-valued modular form which is a lifting of the $\Gamma_0(2)$-modular form $f(\tau)=\eta^{-8}(\tau)\eta^{-8}(2\tau)$. We calculate
\begin{align*}
f&=q^{-1}+8+O(q),\\
f\lvert_S&=16q^{-1/2}+128+O(q^{1/2}).
\end{align*}
Thus the lifting of $f$ for the discriminant form of $\II_{18,2}(2_{\II}^{+n_2})$ gives a nearly holomorphic vector-valued modular form of weight $-4$ whose Borcherds product is a reflective modular form with multiplicities $c_1=1$ and $c_2=2^{(10-n_2)/2}$ and its weight is given by $k=\frac{1}{2}(8+128/16\times 2^{(10-n_2)/2})$. We then finish the construction.

We next give the construction for three exceptional lattices. The reflective modular forms for $2U\oplus 2E_8\oplus D_4$ and $2U\oplus 2E_8\oplus A_2$ were constructed in \cite{Gri18} and \cite{Wan19} as quasi pull-backs $2E_8\oplus D_4 \hookrightarrow 3E_8$ and $2E_8\oplus A_2 \hookrightarrow 3E_8$ of the Borcherds form $\Phi_{12}$. To construct the reflective modular form for $2U\oplus T_8$, we use the obstruction principle of Borcherds product established in \cite{Bor99}. We claim that there is a nearly holomorphic vector-valued modular form of weight $-4$ for the Weil representation of $\SL_2(\ZZ)$ associated to the discriminant form of $2U\oplus T_8$ with principal part 
$$
(q^{-1}+240)\textbf{e}_0 + \sum_{\gamma \in D(T_8), \gamma^2 \equiv \frac{2}{5}\m 2 } 45q^{-1/5} \textbf{e}_\gamma.
$$
The Borcherds product of this modular form gives a reflective modular form of weight $120$ with multiplicities $c_1=1$ and $c_5=45$.  Such a modular form exists if and only if the functional in \cite[Theorem 1.17]{Bru02} equals zero for all holomorphic modular forms of weight $6$ for the dual of the Weil representation. In our case, the dual of the Weil representation associated to the discriminant form of $T_8$ is itself.  By \cite[Corollary 5.5]{Sch15}, the obstruction space is the image of $M_{6}(\Gamma_0(5),\chi)$ under  Scheithauer's lifting, where $M_{6}(\Gamma_0(5),\chi)$ stands for the space of modular forms of weight $6$ for $\Gamma_0(5)$ with the character $\chi=\left(\frac{\cdot}{5} \right)$. Note that $M_{6}(\Gamma_0(5),\chi)$ has dimension $4$ and it is generated by $E_4(\tau)\eta^5(5\tau)/\eta(\tau)$, $E_4(\tau)\eta^5(\tau)/\eta(5\tau)$, $\eta^{15}(5\tau)/\eta^3(\tau)$ and $\eta^{15}(\tau)/\eta^3(5\tau)$. We then assert the existence of the desired vector-valued modular form by direct calculation. 

We explain how to decide whether each modular form is cuspidal or not. Recall that a modular form is cuspidal if and only if it vanishes on all $1$-dimensional and $0$-dimensional cusps. By \cite[Theorems 8.3 and 8.18]{GHS13}, the quasi pull-back increasing the weight is always cuspidal. For the case of $2U\oplus T_8$, it is a maximal lattice and the genus of $T_8$ contains only one class (see \cite{LMFDB}). Thus the associated modular variety has a unique zero-dimensional cusp and a unique one-dimensional cusp. The value of the modular form at the one-dimensional cusp is given by the Siegel operator and is equal to the zeroth Fourier--Jacobi coefficient which is a modular form for $\SL_2(\ZZ)$. We conclude that this $\SL_2(\ZZ)$-modular form is zero from the leading Fourier--Jacobi coefficient of the reflective modular form defined by \eqref{eq:theta}.  Hence the reflective modular form for $2U\oplus T_8$ is cuspidal. We claim that there are $12$ non-cuspidal reflective modular forms. Six of them have singular weight. The others have weight $12$, and a direct calculation shows that their leading Fourier--Jacobi coefficients are not Jacobi cusp forms. Alternatively, we can show that they do not vanish at the $1$-dimensional cusp corresponding to the decomposition $2U\oplus L$ satisfying that the root system associated to $L$ is empty. This proves our claim. 

Finally, we explain why these modular forms cannot be decomposed. If one modular form can be decomposed into a product of $2$-reflective and $2p$-reflective modular forms, then by Lemma \ref{Lem:Sch2p} we have the restriction $n\leq 2 + 24/(p+1)$. Thus we only need to consider a few cases. We prove one case and the proof of other cases is similar. If $U\oplus U(5)\oplus T_4$ is $2$-reflective, then $2U\oplus T_4$ is also $2$-reflective. This leads to a contradiction  by Proposition \ref{Prop:Jacobi} because the sublattice generated by $2$-roots of $T_4$ is $A_2$ which does not span $T_4\otimes \RR$.
\end{proof}

\section{The proof of Theorem \ref{th:main}}\label{Sec:proof}
In this section we prove Theorem \ref{th:main}. 
The rank of reflective lattices of fixed prime level is bounded by \cite[Theorem 4.9]{Wan18} and Lemma \ref{Lem:Dit}. From these bounds, we deduce that reflective lattices of prime level only exist in the following cases: 
  
\begin{enumerate}
\item $p=2$: $n=6$, 10, 14, 18 for all possible lattices; $n=22$ and $n_p=2$.
\item $p=3$: $n=4$, 6, 8, 10, 12, 14 for all possible lattices; $n=20$ and $n_p=1$.
\item $p=5$: $n=6$, $10$ for all possible lattices; $n=14$ and $n_p=1$ or $2$.
\item $p=7$: $n=4$, 6, 8 for all possible lattices; $n=12$ and $n_p=1$.
\item $p=11$: $n=4$, 6 for all possible lattices; $n=8, n_p=1$;  $n=12, n_p=1$.
\item $p\geq 13$ and $p\equiv 1 \m 4$: $n=6, n_p=1$; $n=6, n_p=2$; $n_p=10, n_p=1$.
\item $p=19$:  $n=4$ for all possible lattices; $n=6, n_p=2$; $n=8, n_p=1$.
\item $p=23$:  $n=4$ for all possible lattices; $n=6, n_p=2$; $n=8, n_p=1$.
\item $p> 23$ and $p\equiv 3 \m 4$: $n=4, n_p=1$; $n=6, n_p=2$; $n=8, n_p=1$.
\end{enumerate}

We prove our claim for two cases and the other cases can be proved in a similar way. 
\begin{itemize}
    \item The proof of (1):  By \cite[Theorem 4.9]{Wan18}, it suffices to prove that if $\II_{22,2}(2^{\epsilon_2 n_2})$ is reflective then $n_2=2$. Suppose that $\II_{22,2}(2^{\epsilon_2 n_2})$ is reflective for some $n_2 > 2$. By \S \ref{Sec:basic lemmas}, $n_2$ is even. It is easy to see that this lattice can be represented as $U\oplus U(2)\oplus L$. This contradicts Lemma \ref{Lem:Dit} because $22 > 2 + 48 / (2+1)$. 
    \item The proof of (6): In this case, the lattice $M=\II_{n,2}(p^{\epsilon_p n_p})$ satisfies that $n\in 2 + 4\ZZ$. We assume that $M$ is reflective. By \cite[Theorem 4.9]{Wan18}, we have the bound
    $$
    n \leq 10 + 24/ (p+1),
    $$
    which yields that $n\leq 10$ when $p \geq 13$. When $n=6$, if $n_p\geq 3$ then $M$ can be expressed as $U\oplus U(p) \oplus L$. Then Lemma \ref{Lem:Dit} implies that 
    $$
    6 \leq 2 + 48/ (p+1), \quad \text{i.e.} \quad p \leq 11.
    $$
    This contradicts our assumption $p\geq 13$. When $n=10$,  if $n_p=2$ then $M = U\oplus U(p) \oplus E_8$. Thus $M$ has the form $U\oplus U(p)\oplus L$ when $n=10$ and $n_p\geq 2$. By Lemma \ref{Lem:Dit}, such a lattice is not reflective because $10 \leq 2 + 48/ (p+1)$ does not hold. 
\end{itemize}

In the previous section, we have constructed reflective modular forms for all lattices in Theorem \ref{th:main}. To complete the proof, it suffices to prove that any lattice which appears in the above (1)--(9) and is not listed in Theorem \ref{th:main} is not reflective.

We have known that all lattices formulated in (1) and (2) are reflective. This completes the proof of Theorem  \ref{th:main} for $p=2$, $3$. For the remaining seven cases, we divide the proof into two cases, i.e. the case $p=4x+3$ and the case $p=4x+1$. 

\subsection{The case of p=4x+3}
Let $p=4x+3$ be a prime number larger than 3. Then $p\geq 7$. We have to consider the following cases. 

(I) $n=4$. We only need to consider the cases $n_p=1$ and $n_p=3$. 

When $n_p=1$, the lattice has genus $\II_{4,2}(p^{(-1)^{x+1}})$ and can be represented as $2U\oplus L_p$, where
$$
L_p=\left( \begin{array}{cc}
2 & 1 \\ 
1 & 2x+2
\end{array}  \right).
$$
Suppose that $2U\oplus L_p$ has a symmetric reflective modular form $F$.  By Proposition \ref{Prop:Jacobi},  the root system associated to $L_p$ is $A_1\oplus A_1(p)$. Moreover, the weight and multiplicities of $F$ satisfy 
\begin{align*}
c_p&=pc_1,\\
k&=(35-p)c_1.
\end{align*}
It is clear that $c_1\neq 0$.  The form $F$ is a Borcherds product and we denote its input by $f$.  By the two identities,  the Borcherds product of $f/c_1$ is well defined and gives a reflective modular form whose multiplicity of $2$-reflective divisors is $1$.  Therefore,  we can assume that $c_1=1$.  In view of the singular weight (see \eqref{eq:singular}), we have $k\geq (1+p)/2$, which yields $p\leq 23$. When $p=19$, we have $c_p=19$ and $k=16$. By Scheithauer's condition in \cite[\S 11]{Sch06}, we have 
$$
\frac{3}{k}\cdot \frac{1}{B_{3,\psi}}(p+1)^2=1.
$$
A direct calculation shows that this identity does not hold when $p=19$, a contradiction.

When $n_p=3$, we can represent the lattice as $U\oplus U(p)\oplus L_p$. 
By Lemma \ref{Lem:Sch2p}, $U\oplus U(p)\oplus L_p$ has no $2p$-reflective modular forms when $p=19$ or $p>23$. This follows that the reflective modular form on this lattice (if exists) must have multiplicity $c_1\neq 0$. We conclude from Lemma \ref{rem:Dit} that $U\oplus U(p)\oplus L_p$ is not reflective if $p=19$ or $p>23$, otherwise $2U\oplus L_p$ would be reflective which leads to a contradiction. We have thus proved the following: 

\vspace{3mm}

\textbf{Conclusion.} The lattice $\II_{4,2}(p^{\epsilon_p n_p})$ is not reflective when $p=4x+3$ is $19$ or larger than $23$, where $n_p=1, 3$.

\vspace{3mm}

(II) $n=6$ and $n_p=2$. In this case, the lattice can be written as $2U\oplus 2L_p$. The root system associated to $2L_p$ is $2A_1\oplus 2A_1(p)$. Suppose that $2U\oplus 2L_p$ has a reflective modular form.  Then its weight and multiplicity satisfy 
\begin{align*}
c_p&=pc_1,\\
k&=(34-2p)c_1.
\end{align*}
Similarly, we can assume that $c_1=1$. In view of the singular weight, we have $k\geq 1+p$, which yields $p\leq 11$. This proves the following:

\vspace{3mm}

\textbf{Conclusion.} The lattice $\II_{6,2}(p^{\epsilon_p 2})$ is not reflective when $p=4x+3>11$. 

\vspace{3mm}

(III) $n=8$ and $n_p=1$. In this case, the lattice has genus $\II_{8,2}(p^{(-1)^{x}})$ and we write it as $M_{6,p}=2U\oplus L_{6,p}$. Note that $L_{6,p}$ can be constructed as a maximal even overlattice of $3L_p$ and thus it contains $2$-roots. Suppose that $M_{6,p}$ has a reflective modular form $F$.  By \cite[Proposition 3.2]{Sch06}, there is no vector $v\in M_{6,p}^\vee$ satisfying $\frac{1}{2}(v,v)\equiv \frac{1}{p} \mod 1$. Thus $F$ must be a $2$-reflective modular form and then the root system associated to $L_{6,p}$ is a root lattice of rank $6$ and $ADE$-type. This is impossible when $p>7$ because there is no root lattice of rank $6$ whose determinant is divisible by $p$. We then prove the following:

\vspace{3mm}

\textbf{Conclusion.} The lattice $\II_{8,2}(p^{\epsilon_p 1})$ is not reflective when $p=4x+3>7$. 

\vspace{3mm}

(IV) $p=7$ or $11$.   When $p=7$, we need to prove that the lattice is not reflective if $n=12$ and $n_p=1$.  We choose $2U\oplus E_8\oplus L_7$ as the model of this lattice.  The associated root system of  $E_8\oplus L_7$ defined by Proposition \ref{Prop:Jacobi} is $E_8\oplus A_1\oplus A_1(7)$ (i.e.  $R_1(L)=E_8\oplus A_1$ and $R_2(L)=A_1(7)$).  Since the components $E_8$ and $A_1$ have different Coxeter numbers, we conclude from Proposition \ref{Prop:Jacobi} (3) that $2U\oplus E_8\oplus L_7$ is not reflective.  Similarly,  when $p=11$,  the lattice with $n=12$ and $n_p=1$ can be represented as $2U\oplus E_8 \oplus L_{11}$ and has associated root system $E_8\oplus A_1 \oplus A_1(11)$.  The same argument shows that it is not reflective.

\subsection{The case of p=4x+1}
Let $p=4x+1$ be a prime number and $p\geq 5$. The following lemma is useful for us. 

\begin{lemma}\label{Lem:2-roots}
Let $M$ be a lattice of signature $(6,2)$ and level $p$. Assume that the determinant of $M$ is $p$ or $p^2$. Then there exists a positive definite lattice $L$ containing $2$-roots such that $M=2U\oplus L$.
\end{lemma}

\begin{proof}
We only prove the case of $\det(M)=p^2$ because the other case is similar. We choose a vector $v\in M$ with $(v,v)=2$. It is clear that the ideal generated by $(v,\ell)$, $\ell\in M$, is $\ZZ$. By \cite[Lemma 7.5]{GHS13}, the orthogonal complement $M_v$ of $v$ in $M$ has signature $(5,2)$ and determinant $2p^2$. The minimal number of generators of $M_v^\vee/M_v$ is $2$. By \cite{Nik80} or \cite[Lemma 2.3]{Wan19}, there exists a positive definite lattice $L_1$ of rank $3$ such that $M_v=2U\oplus L_1$. Since $\latt{v}\oplus M_v=2U\oplus\latt{v}\oplus  L_1$ has the overlattice $M$, the lattice $\latt{v}\oplus  L_1$ is not maximal and it has an even overlattice of determinant $p^2$ which is the desired $L$.
\end{proof}

We have to consider the following cases in order to prove Theorem \ref{th:main} for $p=4x+1$.

(I) $n=6$ and $n_p=1$. In this case, the lattice has genus $\II_{6,2}(p^{(-1)^{x+1}})$. By Lemma \ref{Lem:2-roots}, we can choose its model as $M_{4,p}=2U\oplus L_{4,p}$ such that $L_{4,p}$ contains $2$-roots. Suppose that $p>5$ and $M_{4,p}$ has a reflective modular form $F$.  By \cite[Proposition 3.2]{Sch06}, there is no vector $v\in M_{4,p}^\vee$ satisfying $\frac{1}{2}(v,v)\equiv \frac{1}{p} \mod 1$. Thus $F$ must be a $2$-reflective modular form and then the associated root system of $L_{4,p}$ is a root lattice of rank $4$ whose determinant is divisible by $p$. This is impossible when $p>5$, which follows that $M_{4,p}$ is not reflective if $p>5$.

\vspace{3mm}

\textbf{Conclusion.} The lattice $\II_{6,2}(p^{\epsilon_p 1})$ is not reflective if $p=4x+1>5$. 

\vspace{3mm}

(II) $n=6$ and $n_p=2$. In this case, the lattice has genus $\II_{6,2}(p^{-2})$.  Let $N_{4,p}=2U\oplus T_{4,p}$ be its model such that $T_{4,p}$ has $2$-roots.  Suppose that $N_{4,p}$ has a symmetric reflective modular form $F$. The root system associated to $T_{4,p}$ is non-empty and we denote it by $R_1\oplus R_2(p)$. Notice that $S_{4,p}:=(T_{4,p})^\vee(p)$ is an even lattice of level $p$ and determinant $p^2$. We then have $N_{4,p}\cong 2U\oplus S_{4,p}$. The root system associated to $S_{4,p}$ is $R_2\oplus R_1(p)$. Applying Proposition \ref{Prop:Jacobi} to the two models of $N_{4,p}$, we have $h_1=h_2$ and $\rank(R_1)=\rank(R_2)=2$. We denote the Coxeter number of $R_1$ by $h$. Similarly,  we can assume $c_1=1$ and then $c_p=p$. By Proposition \ref{Prop:Jacobi}, the weight of $F$ is given by $k= 12 + (11-p)h$.
By the singular weight argument, we have $k\geq 2(1+p)$. The root lattice $R_1$ has only two possible models: $2A_1$ with $h=2$ and $A_2$ with $h=3$. By direct calculation, there is a contradiction to the weight $k$ if $p\geq 13$. Thus $N_{4,p}$ is not reflective if $p>5$. 

\vspace{3mm}

\textbf{Conclusion.} The lattice $\II_{6,2}(p^{\epsilon_p 2})$ is not reflective if $p=4x+1>5$. 

\vspace{3mm}

(III) $n=10$ and $n_p=1$. In this case, we denote the lattice by $2U\oplus T_{8,p}$.  Suppose that $2U\oplus T_{8,p}$ has a reflective modular form $F$. The lattice $T_{8,p}$ can be constructed as an even overlattice of $E_7\oplus A_1(p)$, in which case the root system associated to $T_{8,p}$ must be $E_7\oplus A_1(p)$. Hence the weight and multiplicity of $F$ satisfy \begin{align*}
c_p&=9pc_1,\\
k&=(165-9p)c_1.
\end{align*}
In view of the singular weight, we have $165-9p\geq (7+9p)/2$, which follows that $p\leq 11$. The only possible case is $p=5$. We have thus proved that $2U\oplus T_{8,p}$ is not reflective if $p>5$. When $p=5$, we have $c_5=45$ and $k=120$, which coincide with our construction for $2U\oplus T_8$ in Theorem \ref{th:construction3}.

\vspace{3mm}

\textbf{Conclusion.} The lattice $\II_{10,2}(p^{\epsilon_p 1})$ is not reflective if $p=4x+1>5$. 

\vspace{3mm}

(IV) $p=5$.
It remains to consider the three cases: $n=10$ and $n_p=3$, $n=14$ and $n_p=1$ or $2$. For the first case, we notice that its model $2U\oplus A_4 \oplus T_4$ has associated root system $A_4\oplus A_2 \oplus A_2(5)$. Since $A_4$ and $A_2$ have different Coxeter numbers, we deduce from Proposition \ref{Prop:Jacobi} that $2U\oplus A_4 \oplus T_4$ is not reflective. For the last two cases, we prove that they are not reflective in a similar way using models $2U\oplus E_8\oplus A_4$ and $2U\oplus E_8\oplus T_4$.

\begin{remark}
It is an interesting question if $T_{4,p}$ and $(T_{4,p})^\vee(p)$ are isomorphic as lattices. When $p=5$, we know from \cite{LMFDB} that the genus of $T_{4,5}$ has a unique class $T_4$. Thus this question has a positive answer when $p=5$.  But we do not know the answer in general case.
\end{remark}

\section{Applications}\label{Sec:applications}
\subsection{The Kodaira dimension of orthogonal modular varieties}
In this subsection we use our reflective modular forms to determine the Kodaira dimension of some orthogonal modular varieties. By \cite[Theorem 2.1]{GH14}, if there is a strongly reflective modular form of large weight, then the modular variety is uniruled and thus has Kodaira dimension $-\infty$. 

Applying the Gritsenko--Hulek criterion to reflective modular forms in Theorem \ref{th:construction1}, we have

\begin{theorem}\label{Th6.1}
The orthogonal modular variety $\cD(M)/ \widetilde{\Orth}^+(M)$ is uniruled if $M$ is in one of the genera formulated in Theorem \ref{th:construction1}.
\end{theorem}

Applying the Gritsenko--Hulek criterion to reflective modular forms in Theorem \ref{th:construction3}, we have

\begin{theorem}\label{Th6.2}
The orthogonal modular variety $\cD(M)/ \Orth^+(M)$ is uniruled if $M$ is in one of the following genera
\begin{align*}
&\II_{14,2}(2_{\II}^{-2})&
&\II_{14,2}(2_{\II}^{-4})& 
&\II_{14,2}(2_{\II}^{-6})&  
&\II_{14,2}(2_{\II}^{-8})& \\
&\II_{10,2}(3^{+4})& 
&\II_{10,2}(3^{-6})& 
&\II_{6,2}(5^{-4})& 
&\II_{4,2}(7^{-3}).& 
\end{align*}
\end{theorem}

As one referee pointed out, the five cases with $p=2$ in Theorem \ref{Th6.1} are in fact unirational by \cite[Theorem 3.8]{Ma12}, and the first two cases with $p=2$ in Theorem \ref{Th6.2} are in fact rational by \cite{Ma15}.

By \cite[Theorem 1.5]{Gri18}, if there is a strongly reflective cusp form of canonical weight, then the Kodaira dimension of the modular variety is zero. Note that the modular variety $\cD(M)/\Orth^+(M)$ is a quasi-projective variety of dimension $n$ when $M$ has signature $(n,2)$, in which case the canonical weight is $n$.  Strongly reflective modular forms of canonical weight are very rare. There are only four examples in the literature. The underlying lattices are respectively $U\oplus U(2)\oplus A_1$, $2U\oplus 2A_1$, $2U\oplus 2A_2$, $2U\oplus D_5$ (see \cite[Theorem 4.1]{CG11} and \cite[Theorem 5.8]{Gri18} ). We here give two new such exceptional modular forms. 

\begin{theorem}
Let $\Orth_r(M)$ denote the subgroup of $\Orth^+(M)$ generated by reflections. When $M$ is in one of the two genera
\begin{align*}
&\II_{12,2}(3^{+7}),& &\II_{4,2}(11^{+3}),&
\end{align*}
the Kodaira dimension of $\cD(M)/ \Orth_r(M)$ is zero. Moreover, the geometric genus of $\cD(M)/ \Orth_r(M)$ is one. 
\end{theorem}
\begin{proof}
We use the reflective modular forms constructed in Theorem \ref{th:construction3}.  The proof is similar to that of  \cite[Theorem 4.1]{CG11}. We only need to show that the character of the reflective modular form on $\Orth_r(M)$ is $\det$. Since the modular form vanishes on all reflective $\gamma^\perp$ with multiplicity one and the reflections $\sigma_\gamma$ lie in $\Orth_r(M)$ which is the modular group,  the value of the character on every reflection is $-1$ (considering the Taylor expansion of the modular form about $\gamma^\perp$). Thus the character is the determinant. 
\end{proof}

For the above two lattices,  by \cite[Theorem 1.1, Corollary 1.2]{GHS09}, $\widetilde{\Orth}^+(M)$ is generated by $2$-reflections, which implies that $\widetilde{\Orth}^+(M)< \Orth_r(M)$.  By \cite[Corollary 5.4]{GHS09}, the modular variety $\cD(M)/ \Orth_r(M)$ is simply connected. It would be interesting to know if $\Orth^+(M)$ is generated by reflections, namely $\Orth_r(M)=\Orth^+(M)$. The answer is expected to be positive.

\subsection{The genera of lattices} Our method can also be used to calculate the class number of the genus of a lattice. Let $L$ be an even positive definite lattice and $M=2U\oplus L$. Suppose that there is a reflective modular form $F$ for $M$. For every class $L_1$ in the genus of $L$, the modular form $F$ has a Fourier--Jacobi expansion at the one-dimensional cusp determined by the decomposition $M\cong 2U\oplus L_1$. Thus there is a Jacobi form of weight $0$ and index $L_1$ whose Borcherds product equals $F$. By Theorem \ref{Prop:Jacobi}, the root system  of $L_1$ satisfies either $R(L_1)=\emptyset$ or $\rank(R(L_1))=\rank(L_1)$. In the latter case, it is possible to determine $L_1$ because it is an overlattice of $R(L_1)$. As an example, we compute the class number of $\II_{10,0}(3^{+7})$.

\begin{proposition}
The genus $\II_{10,0}(3^{+7})$ has two classes. The first one has model $E_6^\vee(3)\oplus 2A_2$ and the second can be constructed as an even overlattice of $A_3\oplus D_7(3)$.
\end{proposition}

\begin{proof}
Let $M$ be the unique class of $\II_{12,2}(3^{+7})$. We have constructed a strongly reflective cusp form $F$ of weight $12$ for $\Orth^+(M)$. Let $L$ be a class of $\II_{10,0}(3^{+7})$. Then $M\cong 2U\oplus L$. By Theorem \ref{Prop:Jacobi}, the root system $R(L)$ is 
not empty, otherwise $F$ has a norm zero Weyl vector at the one-dimensional cusp related to $L$. Thus $R(L)$ is a root system of rank $10$. Since $R(L)$ has only $2$- and $6$-reflections, it is a direct sum of some irreducible components of types $G_2$, $R$ and $R(3)$, where $R$ is a root system of $ADE$-type. All irreducible components of $R(L)$ have the same Coxeter number defined as the constant $C$ in the identity of type \eqref{eq:two}. The Coxeter numbers of $G_2$, $R$ and $R(3)$ are respectively $4$, $h_0$ and $h_0/3$, where $h_0$ is the usual Coxeter number of $R$. Let $a$ and $b$ be the number of $2$-roots and $6$-roots in $R(L)$, respectively. By \eqref{eq:one}, we have
\begin{align*}
\frac{1}{24}(a+b+24)-1&=h=\frac{1}{20}\left(2a+\frac{2}{3}b\right),
\end{align*}
which follows $b=7a$ and $h=a/3$, where $h$ is the Coxeter number of $R(L)$.
By direct calculation,  there are exactly two possibilities $R(L)=E_6(3)\oplus 2G_2$ or $A_3\oplus D_7(3)$. It is easy to see that the two root systems have  a unique even overlattice of determinant $3^7$ and level $3$ respectively. We then complete the proof.
\end{proof}

\bigskip

\noindent
\textbf{Acknowledgements} 
The work was done when the author was a postdoc at the Max Planck Institute for Mathematics in Bonn.  The author thanks the institute for its hospitality and financial support.  The author was also supported by the Institute for Basic Science (IBS-R003-D1).  The author thanks Valery Gritsenko for helpful discussions and the two referees for valuable comments.

\bibliographystyle{plainnat}
\bibliofont
\bibliography{refs}

\begin{thebibliography}{31}
\providecommand{\natexlab}[1]{#1}
\providecommand{\url}[1]{\texttt{#1}}
\expandafter\ifx\csname urlstyle\endcsname\relax
  \providecommand{\doi}[1]{doi: #1}\else
  \providecommand{\doi}{doi: \begingroup \urlstyle{rm}\Url}\fi

\bibitem[LMF()]{LMFDB}
{The $L$-functions and modular forms database-integral lattice}.
\newblock \url{http://www.lmfdb.org/Lattice/}.

\bibitem[Borcherds(1998)]{Bor98}
Richard Borcherds.
\newblock Automorphic forms with singularities on {G}rassmannians.
\newblock \emph{Invent. Math.}, 132\penalty0 (3):\penalty0 491--562, 1998.

\bibitem[Borcherds(1995)]{Bor95}
Richard~E. Borcherds.
\newblock Automorphic forms on {${\rm O}_{s+2,2}({\bf R})$} and infinite
  products.
\newblock \emph{Invent. Math.}, 120\penalty0 (1):\penalty0 161--213, 1995.

\bibitem[Borcherds(1999)]{Bor99}
Richard~E. Borcherds.
\newblock The {G}ross-{K}ohnen-{Z}agier theorem in higher dimensions.
\newblock \emph{Duke Math. J.}, 97\penalty0 (2):\penalty0 219--233, 1999.

\bibitem[Borcherds(2000)]{Bor00}
Richard~E. Borcherds.
\newblock Reflection groups of {L}orentzian lattices.
\newblock \emph{Duke Math. J.}, 104\penalty0 (2):\penalty0 319--366, 2000.

\bibitem[Borcherds et~al.(1998)Borcherds, Katzarkov, Pantev, and
  Shepherd-Barron]{BKP98}
Richard~E. Borcherds, Ludmil Katzarkov, Tony Pantev, and N.~I. Shepherd-Barron.
\newblock Families of {$K3$} surfaces.
\newblock \emph{J. Algebraic Geom.}, 7\penalty0 (1):\penalty0 183--193, 1998.

\bibitem[Bourbaki(1968)]{Bou60}
N.~Bourbaki.
\newblock \emph{\'{E}l\'{e}ments de math\'{e}matique. {F}asc. {XXXIV}.
  {G}roupes et alg\`ebres de {L}ie. {C}hapitre {IV}: {G}roupes de {C}oxeter et
  syst\`emes de {T}its. {C}hapitre {V}: {G}roupes engendr\'{e}s par des
  r\'{e}flexions. {C}hapitre {VI}: syst\`emes de racines}.
\newblock Actualit\'{e}s Scientifiques et Industrielles [Current Scientific and
  Industrial Topics], No. 1337. Hermann, Paris, 1968.

\bibitem[Bruinier(2002)]{Bru02}
Jan~H. Bruinier.
\newblock \emph{Borcherds products on {O}(2, {$l$}) and {C}hern classes of
  {H}eegner divisors}, volume 1780 of \emph{Lecture Notes in Mathematics}.
\newblock Springer-Verlag, Berlin, 2002.

\bibitem[Bruinier(2014)]{Bru14}
Jan~Hendrik Bruinier.
\newblock On the converse theorem for {B}orcherds products.
\newblock \emph{J. Algebra}, 397:\penalty0 315--342, 2014.

\bibitem[Cl\'{e}ry and Gritsenko(2011)]{CG11}
F.~Cl\'{e}ry and V.~Gritsenko.
\newblock Siegel modular forms of genus 2 with the simplest divisor.
\newblock \emph{Proc. Lond. Math. Soc. (3)}, 102\penalty0 (6):\penalty0
  1024--1052, 2011.

\bibitem[Conway and Sloane(1999)]{CS99}
J.~H. Conway and N.~J.~A. Sloane.
\newblock \emph{Sphere packings, lattices and groups}, volume 290 of
  \emph{Grundlehren der mathematischen Wissenschaften [Fundamental Principles
  of Mathematical Sciences]}.
\newblock Springer-Verlag, New York, third edition, 1999.
\newblock With additional contributions by E. Bannai, R. E. Borcherds, J.
  Leech, S. P. Norton, A. M. Odlyzko, R. A. Parker, L. Queen and B. B. Venkov.

\bibitem[Dittmann(2019)]{Dit18}
Moritz Dittmann.
\newblock Reflective automorphic forms on lattices of squarefree level.
\newblock \emph{Trans. Amer. Math. Soc.}, 372\penalty0 (2):\penalty0
  1333--1362, 2019.

\bibitem[Eichler and Zagier(1985)]{EZ85}
Martin Eichler and Don Zagier.
\newblock \emph{The theory of {J}acobi forms}, volume~55 of \emph{Progress in
  Mathematics}.
\newblock Birkh\"{a}user Boston, Inc., Boston, MA, 1985.

\bibitem[Gritsenko and Hulek(2014)]{GH14}
V.~Gritsenko and K.~Hulek.
\newblock Uniruledness of orthogonal modular varieties.
\newblock \emph{J. Algebraic Geom.}, 23\penalty0 (4):\penalty0 711--725, 2014.

\bibitem[Gritsenko et~al.(2009)Gritsenko, Hulek, and Sankaran]{GHS09}
V.~Gritsenko, K.~Hulek, and G.~K. Sankaran.
\newblock Abelianisation of orthogonal groups and the fundamental group of
  modular varieties.
\newblock \emph{J. Algebra}, 322\penalty0 (2):\penalty0 463--478, 2009.

\bibitem[Gritsenko et~al.(2013)Gritsenko, Hulek, and Sankaran]{GHS13}
V.~Gritsenko, K.~Hulek, and G.~K. Sankaran.
\newblock Moduli of {K}3 surfaces and irreducible symplectic manifolds.
\newblock In \emph{Handbook of moduli. {V}ol. {I}}, volume~24 of \emph{Adv.
  Lect. Math. (ALM)}, pages 459--526. Int. Press, Somerville, MA, 2013.

\bibitem[Gritsenko(2018)]{Gri18}
V.~A. Gritsenko.
\newblock Reflective modular forms and their applications.
\newblock \emph{Uspekhi Mat. Nauk}, 73\penalty0 (5(443)):\penalty0 53--122,
  2018.

\bibitem[Gritsenko and Vang(2019)]{GW19}
V.~A. Gritsenko and Kh. Vang.
\newblock Weight 3 antisymmetric paramodular forms.
\newblock \emph{Mat. Sb.}, 210\penalty0 (12):\penalty0 43--66, 2019.

\bibitem[Gritsenko et~al.(2007)Gritsenko, Hulek, and Sankaran]{GHS07}
V.~A. Gritsenko, K.~Hulek, and G.~K. Sankaran.
\newblock The {K}odaira dimension of the moduli of {$K3$} surfaces.
\newblock \emph{Invent. Math.}, 169\penalty0 (3):\penalty0 519--567, 2007.

\bibitem[Gritsenko and Nikulin(1998)]{GN98}
Valeri~A. Gritsenko and Viacheslav~V. Nikulin.
\newblock Automorphic forms and {L}orentzian {K}ac-{M}oody algebras. {II}.
\newblock \emph{Internat. J. Math.}, 9\penalty0 (2):\penalty0 201--275, 1998.

\bibitem[Gritsenko and Nikulin(2018)]{GN18}
Valery Gritsenko and Viacheslav~V. Nikulin.
\newblock Lorentzian {K}ac-{M}oody algebras with {W}eyl groups of
  2-reflections.
\newblock \emph{Proc. Lond. Math. Soc. (3)}, 116\penalty0 (3):\penalty0
  485--533, 2018.

\bibitem[Ma(2012)]{Ma12}
Shouhei Ma.
\newblock The unirationality of the moduli spaces of 2-elementary {$K3$}
  surfaces.
\newblock \emph{Proc. Lond. Math. Soc. (3)}, 105\penalty0 (4):\penalty0
  757--786, 2012.
\newblock With an appendix by Ken-Ichi Yoshikawa.

\bibitem[Ma(2015)]{Ma15}
Shouhei Ma.
\newblock Rationality of the moduli spaces of 2-elementary {$K3$} surfaces.
\newblock \emph{J. Algebraic Geom.}, 24\penalty0 (1):\penalty0 81--158, 2015.

\bibitem[Ma(2017)]{Ma17}
Shouhei Ma.
\newblock Finiteness of 2-reflective lattices of signature {$(2,n)$}.
\newblock \emph{Amer. J. Math.}, 139\penalty0 (2):\penalty0 513--524, 2017.

\bibitem[Ma(2018)]{Ma18}
Shouhei Ma.
\newblock On the {K}odaira dimension of orthogonal modular varieties.
\newblock \emph{Invent. Math.}, 212\penalty0 (3):\penalty0 859--911, 2018.

\bibitem[Nikulin(1979)]{Nik80}
V.~V. Nikulin.
\newblock Integer symmetric bilinear forms and some of their geometric
  applications.
\newblock \emph{Izv. Akad. Nauk SSSR Ser. Mat.}, 43\penalty0 (1):\penalty0
  111--177, 238, 1979.

\bibitem[Scheithauer(2006)]{Sch06}
Nils~R. Scheithauer.
\newblock On the classification of automorphic products and generalized
  {K}ac-{M}oody algebras.
\newblock \emph{Invent. Math.}, 164\penalty0 (3):\penalty0 641--678, 2006.

\bibitem[Scheithauer(2015)]{Sch15}
Nils~R. Scheithauer.
\newblock Some constructions of modular forms for the {W}eil representation of
  {${\rm SL}_2(\mathbb{Z})$}.
\newblock \emph{Nagoya Math. J.}, 220:\penalty0 1--43, 2015.

\bibitem[Scheithauer(2017)]{Sch17}
Nils~R. Scheithauer.
\newblock Automorphic products of singular weight.
\newblock \emph{Compos. Math.}, 153\penalty0 (9):\penalty0 1855--1892, 2017.

\bibitem[Wang(2019)]{Wan19}
Haowu Wang.
\newblock The classification of 2-reflective modular forms.
\newblock Prepint, 2019.
\newblock URL \url{arXiv:1906.10459}.

\bibitem[Wang(2021)]{Wan18}
Haowu Wang.
\newblock Reflective modular forms: a {J}acobi forms approach.
\newblock \emph{Int. Math. Res. Not. IMRN}, \penalty0 (3):\penalty0 2081--2107,
  2021.

\end{thebibliography}

\end{document}